\documentclass[11pt, a4paper]{article}
\usepackage{amsfonts}
\usepackage{}
\usepackage{amsthm}
\usepackage{amsmath,amssymb,latexsym,color}
\usepackage[mathscr]{eucal}
\usepackage[colorlinks,
            linkcolor=blue,
            anchorcolor=green,
            citecolor=magenta
           ]{hyperref}
\usepackage{graphicx}
\usepackage{tikz}
\usetikzlibrary{calc}
\usepackage{authblk}

\long\def\delete#1{}

\usepackage{color}

\definecolor{Blue}{rgb}{0,0,1}
\definecolor{Red}{rgb}{1,0,0}
\definecolor{DarkGreen}{rgb}{0,0.6,0}
\definecolor{DarkYellow}{rgb}{1,1,0.2}
\definecolor{DarkPurple}{rgb}{.6,0,1}

\usepackage{xcolor}
\usepackage[normalem]{ulem}

\usepackage{cleveref}
\crefformat{section}{\S#2#1#3}
\crefformat{subsection}{\S#2#1#3}
\crefformat{subsubsection}{\S#2#1#3}
\crefrangeformat{section}{\S\S#3#1#4 to~#5#2#6}
\crefmultiformat{section}{\S\S#2#1#3}{ and~#2#1#3}{, #2#1#3}{ and~#2#1#3}

\newcommand{\CC}{{\mathcal {C}}} 
 \newcommand{\CF}{{\mathcal {F}}}
\newcommand{\CG}{{\mathcal {G}}}

 \newcommand{\CT}{{\mathcal {T}}}

\newcommand{\abs}[1]{\left | #1 \right |}

\newcommand{\ls}{\leqslant}
\newcommand{\gs}{\geqslant}
\newcommand{\curlybraces}[1]{\left\{ #1 \right\}}
\newcommand{\pth}[1]{\left(#1\right)}
\newcommand{\sqb}[1]{\left[#1\right]}

\begin{document}
\setcounter{page}{1}
\newtheorem{thm}{Theorem}[section]
\newtheorem{fthm}[thm]{Fundamental Theorem}
\newtheorem{dfn}[thm]{Definition}
\newtheorem*{rem}{Remark}
\newtheorem{lem}[thm]{Lemma}
\newtheorem{cor}[thm]{Corollary}
\newtheorem{exa}[thm]{Example}
\newtheorem{prop}[thm]{Proposition}
\newtheorem{prob}[thm]{Problem}
\newtheorem{fact}[section]{Fact}
\newtheorem{con}[thm]{Conjecture}
\renewcommand{\thefootnote}{}
\newcommand{\remark}{\vspace{2ex}\noindent{\bf Remark.\quad}}
\newtheorem{ob}[thm]{Observation}
\newcommand{\rmnum}[1]{\romannumeral #1}
\renewcommand{\abovewithdelims}[2]{%
\genfrac{[}{]}{0pt}{}{#1}{#2}}

\newcommand\Sy{\mathrm{S}}
\newcommand\Cay{\mathrm{Cay}}
\newcommand\tw{\mathrm{tw}}
\newcommand\supp{\mathrm{supp}}


\def\qed{\hfill$\Box$\vspace{11pt}}

\title {\bf  Triangles in $ r $-wise $ t $-intersecting families
\thanks{
This work is  supported by the National Natural Science Foundation of China (Grant 12171272, 12161141003, 11971158)}}
\author{Jiaqi Liao\thanks{ E-mail: \texttt{liaojq19@mails.tsinghua.edu.cn}}}
\author{Mengyu Cao\thanks{Corresponding author. E-mail: \texttt{caomengyu@mail.bnu.edu.cn}}}

\author{Mei Lu\thanks{E-mail: \texttt{lumei@tsinghua.edu.cn}}}

\affil{\small Department of Mathematical Sciences, Tsinghua University, Beijing 100084, China}

\date{}

\openup 0.5\jot
\maketitle

\begin{abstract}
Let $t$, $r$, $k$ and $n$ be  positive integers and $\mathcal{F}$  a family of $k$-subsets of an $n$-set $V$. The family $ \CF $  is $ r $-wise $ t $-intersecting if for any $ F_1, \ldots, F_r \in \CF $, we have $ \abs{\cap_{i = 1}^{r}F_i}\gs t $. An $ r $-wise $ t $-intersecting family of $ r + 1 $ sets $ \{T_1, \ldots, T_{r + 1}\} $ is called an $ (r + 1,t) $-triangle if $ |T_1 \cap \cdots \cap T_{r + 1}| \ls t - 1 $. In this paper, we prove that if $ n \gs n_0(r, t, k) $, then the $ r $-wise $ t $-intersecting family $ \CF \subseteq \binom{[n]}{k} $ containing the most $ (r + 1,t) $-triangles is isomorphic to $ \curlybraces{F \in \binom{[n]}{k}: \abs{F \cap [r + t]} \gs r + t - 1} $. This can also be regarded as a generalized Tur\'{a}n type result.

\vspace{2mm}

\noindent{\bf Key words}\ \ $ r $-wise $ t $-intersecting family, triangle, generalized Tur\'{a}n type problem

\

\noindent{\bf MSC2010:} \   05C36, 52A37, 05A30

\end{abstract}

\section{Introduction}
Throughout this paper, we fix three integers $ r $, $ t $ and $ k $ with $ r \gs 2 $, $ t \gs 1 $ and $ t \ls k - r $. By sets we mean finite sets. For a positive integer $ n $, denote $ [n] := \{1, 2, \ldots, n\} $ and $ [0] := \emptyset $. For two positive integers $ n_1 $ and $ n_2 $ with $1\ls n_1<n_2$, denote $ [n_1, n_2] := [n_2] \setminus [n_1 - 1] $. For $ 0 \ls k \ls n $, denote $ \binom{[n]}{k} $ as the family of all $ k $-element subsets of $ [n] $. In this paper, we only consider $ k $-uniform family $ \CF \subseteq \binom{[n]}{k} $. A $ k $-uniform family $ \CF $ is $ r $-\emph{wise} $ t $-\emph{intersecting} if for any $ F_1, \ldots, F_r \in \CF $, $ |F_1 \cap \cdots \cap F_r| \gs t $ holds. Observe that the `$2$-wise $t$-intersecting' is the classical `$t$-intersecting'.
Two families $ \CF_1, \CF_2 \subseteq \binom{[n]}{k} $ are \emph{isomorphic} if $ \CF_1 $ can be obtained from $ \CF_2 $ by a permutation of $ [n] $, denoted by $ \CF_1 \cong \CF_2 $. An $ r $-wise $ t $-intersecting family $ \CF $ is \emph{trivial} if $ \CF \cong \curlybraces{F \in \binom{[n]}{k}: [t] \subseteq F} $ and \emph{non-trivial} otherwise. An $ r $-wise $ t $-intersecting family $ \CF $ is \emph{maximal} if adding another $ k $-element subset to $ \CF $ will lose $ r $-wise $ t $-intersecting property. In $k$-uniform intersecting family $ \CF \subseteq \binom{[n]}{k} $, a triple $\{F_1,F_2,F_3\}$ in $\CF$ is called \emph{triangle} if $|F_1\cap F_2\cap F_3|=0$.

The study of intersecting families has been going on for more than 60 years, starting with the famous Erd\H{o}s-Ko-Rado theorem (EKR theorem for short) \cite{Erdos-Ko-Rado-1961-313}, which gives the maximum size of an $ t $-intersecting family and shows further that each $ t $-intersecting family with maximum size is trivial for sufficiently large $ n $. After that,  the EKR theorem was generalized in many directions, for example, after the efforts of Hilton, Milner, Frankl, F\"{u}redi, Ahlswede and Khachatrian, the structure of the maximum non-trivial $ t $-intersecting families is now completely clarified. Readers are invited to refer \cite{AK1996, Ahlswede-Khachatrian-1997, Frankl-1978, Frankl1978, FZ1986, HM1967}. There existed  results  concerning uniform intersecting families \cite{FK, FKK, FKK2}. In this paper, we focus on  $k$-uniform $t$-intersecting families.

Tur\'{a}n type problems is one of the most studied areas in extremal graph theory. They
ask for the maximum size of a combinatorial structure that avoids some forbidden structure. A generalized Tur\'{a}n type problem has attracted lots of attentions recently, which is the maximum number of copies of a fixed graph $F$ among $n$-vertex $H$-free graphs $G$. A rapidly growing literatures studies this type of problems.


In \cite{NP2022}, Nagy and Patk\'{o}s  defined a collection $\CT$  of $(r+1)$ sets is an $(r+1)$-\emph{triangle}  if for every $T_1,T_2,\ldots,T_{r}\in \CT$ we have $T_1 \cap \cdots \cap T_{r}\neq \emptyset$, but  $\cap_{T\in \CT}T=\emptyset$.  And they discussed the structure of the $r$-wise intersecting family $\CF$ containing the most number of $(r+1)$-triangles. Denote $\CG_{r,1} = \curlybraces{F \in \binom{[n]}{k}: |F \cap [r + 1]| = r }$ and $\CG'_{r,1} = \curlybraces{F \in \binom{[n]}{k}: \abs{F \cap [r + 1]} \gs r }.$
Write $n_{r+1,k}$ to denote the number of
$(r + 1)$-triangles in  $\CG_{r,1}$. In \cite{NP2022}, the following two results were
  obtained.

\begin{thm} {\rm \cite{NP2022}}\label{npthm1}
	For  every $ k\gs 2$, if   $ \CF \subseteq \binom{[n]}{k} $ is  intersecting  with $ n \gs 4k^6 $, then the number of triangles in $ \CF$ is at most $n_{3,k}$ and equality holds if and only if $ \CG_{2,1} \subseteq \CF \subseteq \CG'_{2,1} $ up to isomorphism.
\end{thm}

\begin{thm} {\rm \cite{NP2022}}\label{npthm2}
	For  every $ k\gs r\gs 3$, if   $ \CF \subseteq \binom{[n]}{k} $ is  $r$-wise intersecting  with $ n \gs 4k^{r(r+1)} $, then the number of triangles in $ \CF$ is at most $n_{r+1,k}$ and equality holds if and only if $ \CG_{r,1} \subseteq \CF \subseteq \CG'_{r,1} $ up to isomorphism
\end{thm}


We generalized the definition of $(r+1)$-triangle to $ (r + 1,t) $-triangle.  An $ r $-wise $ t $-intersecting family of $ r + 1 $ sets $ \{T_1, \ldots, T_{r + 1}\} $ is called an $ (r + 1,t) $-\emph{triangle} if $ |T_1 \cap \cdots \cap T_{r + 1}| \ls t - 1 $. Then $ (r + 1,1) $-triangle is $ (r + 1) $-triangle. By the definition, an $ r $-wise $ t $-intersecting family is not necessarily an $ (r + 1) $-wise $ t $-intersecting family. So it is natural to ask which $ r $-wise $ t $-intersecting family contains the most $ (r + 1, t) $-triangles. Since trivial $ r $-wise $ t $-intersecting families have no $ (r + 1,t) $-triangles, we only consider non-trivial $ r $-wise $ t $-intersecting families. For any $ r $-wise $ t $-intersecting family $ \CF $, we write $ N_{r+1,t}(\CF) $ to denote the number of $ (r + 1,t) $-triangles in $ \CF $. We may assume $ \CF $ is maximal since $ N_{r+1,t}(\CF) $ will not get smaller by adding $ k $-element sets to $ \CF $. It is known that each maximal non-trivial $ k $-uniform $ (k - 1) $-intersecting family is isomorphic to $ \binom{[k + 1]}{k} $, and by Lemma \ref{gen_of_prop2.2} in Section 4, an $ r $-wise $ t $-intersecting family is $ (r + t - 2) $-intersecting, thus we only consider the case  $ 1 \ls t \ls k - r $. Denote
$$\begin{array}{rcl}
\CG_{r,t} &=& \curlybraces{F \in \binom{[n]}{k}: |F \cap [r + t]| = r + t - 1}, \\
 \CG'_{r,t} &=& \curlybraces{F \in \binom{[n]}{k}: \abs{F \cap [r + t]} \gs r + t - 1}.\end{array}$$
Then $N_{r+1,1}(\CG)=n_{r+1,k}$.

In this paper, we consider the structure of the $r$-wise $t$-intersecting family $\CF$ containing the most number of $(r+1,t)$-triangles and obtain the following results which generalized Theorems \ref{npthm1} and \ref{npthm2}, respectively.

\begin{thm}\label{maintheorem}
	
Let $k\gs 3$, $ 1 \ls t \ls k - 2 $ and  $ \CF \subseteq \binom{[n]}{k} $ be a $ t $-intersecting family with $ n \gs k^4 $. Then $ N_{3,t}(\CF) \ls N_{3,t}(\CG_{2,t}) $ and equality holds if and only if $ \CG_{2,t} \subseteq \CF \subseteq \CG'_{2,t} $ up to isomorphism.

\end{thm}


\begin{thm}\label{maintheorem_gen}
	
Let $k>r\gs 3$, $ 1 \ls t \ls k - r $ and  $ \CF \subseteq \binom{[n]}{k} $ be an $ r $-wise $ t $-intersecting family. Then there exist two positive constants $ c = c(r, t) $ and $ d = d(r, t) $, such that $ N_{r+1,t}(\CF) \ls N_{r+1,t}(\CG_{r,t}) $ for $ n \gs c k^d $, and equality holds if and only if $ \CG_{r,t} \subseteq \CF \subseteq \CG'_{r,t} $ up to isomorphism.
	
\end{thm}

The rest of this paper is organized as follows. In Section \ref{pre} we will give some useful lemmas which will be used in proving our main results. We will prove Theorem \ref{maintheorem} in Section \ref{proof}.  Theorem \ref{maintheorem_gen} will be proved in Section \ref{section_gen}.

\section{Preliminaries}\label{pre}

In this section, we present some useful statements that we will need during the proofs.

\begin{lem}
	
Let $ r \gs 2 $, $ 1 \ls t \ls k - r $ and $ n \gs k^4 $. Then we have

\[N_{r+1,t}(\CG'_{r,t}) = N_{r+1,t}(\CG_{r,t}) \gs \frac{999}{1000}\binom{r + t}{r + 1}\binom{n - r - t}{k - r - t + 1}^{r + 1}.\]	

\end{lem}

\begin{proof}
	
Obviously $N_{r+1,t}(\CG'_{r,t}) \gs N_{r+1,t}(\CG_{r,t})$. Note that $ \CG_{r,t} $ has a decomposition, say $ \CG_{r,t} = \cup_{i = 1}^{r + t} \CG_i $, where $ \CG_i := \curlybraces{F \in \CG_{r,t}: i \notin F} $. Then $ \abs{\CG_i} = \binom{n - r - t}{k - r - t + 1} $. Any $ (r + 1,t) $-triangle in $ \CG'_{r,t} $ must have the  form $ \{F_{i_1}, \ldots, F_{i_{r + 1}}\} $, where $ i_1, \ldots, i_{r + 1} \in [r + t] $ are $ r + 1 $ pairwise distinct indices, and $ (F_{i_1}, \cdots, F_{i_{r + 1}}) \in \CG_{i_1} \times \cdots \times \CG_{i_{r + 1}} $; otherwise $ \abs{F_{i_1} \cap \ldots \cap F_{i_{r + 1}}} \gs t $ and we get a contradiction. It implies $ N_{r+1,t}(\CG'_{r,t}) = N_{r+1,t}(\CG_{r,t}) $. On the other hand, for $ r + 1 $ pairwise distinct indices $ \{i_1, \ldots, i_{r + 1}\} \subseteq [r + t] $, we have
\begin{align*}
&\curlybraces{(F_{i_1}, \cdots, F_{i_{r + 1}}) \in \CG_{i_1} \times \cdots \times \CG_{i_{r + 1}}: \abs{F_{i_1} \cap \cdots \cap F_{i_{r + 1}}} \ls t - 1} \\
=& \CG_{i_1} \times \cdots \times \CG_{i_{r + 1}} \setminus \bigcup_{\ell = r + t + 1}^n\curlybraces{(F_{i_1}, \cdots, F_{i_{r + 1}}) \in \CG_{i_1} \times \cdots \times \CG_{i_{r + 1}}: \ell \in F_{i_1} \cap \cdots \cap F_{i_{r + 1}}}.
\end{align*}

\noindent Since
\begin{align*}
& \abs{\bigcup_{\ell = r + t + 1}^n\curlybraces{(F_{i_1}, \ldots, F_{i_{r + 1}}) \in \CG_{i_1} \times \cdots \times \CG_{i_{r + 1}}: \ell \in F_{i_1} \cap \cdots \cap F_{i_{r + 1}}}}\\
 \ls &(n - r - t)\binom{n - r - t - 1}{k - r - t}^{r + 1},\end{align*}
we have
\begin{align*}
N_{r+1,t}(\CG_{r,t}) &\gs \binom{r + t}{r + 1}\sqb{\binom{n - r - t}{k - r - t + 1}^{r + 1} - (n - r - t)\binom{n - r - t - 1}{k - r - t}^{r + 1}} \\
&= \pth{1 - \frac{(k - r - t + 1)^{r + 1}}{(n - r - t)^r}}\binom{r + t}{r + 1}\binom{n - r - t}{k - r - t + 1}^{r + 1}.
\end{align*}

\noindent Set $ f(x) := 1 - \frac{(k - r + 1 - x)^{r + 1}}{(n - r - x)^r} $, then
\[f'(x) = -\frac{(-r + 1 + k - x)^r (r(k + 2) - (r + 1)n + x)}{(-r + n - x)^{r + 1}} \gs 0\]

\noindent for $ 1 \ls x \ls k - r $. Hence

\[1 - \frac{(k - r - t + 1)^{r + 1}}{(n - r - t)^r} \gs 1- \frac{(k - r)^{r+1}}{(n - r - 1)^r} \gs 1 - \frac{(k - r)^{r + 1}}{(k^4 - r - 1)^r},\]

\noindent Set $ g(y) := 1 - \frac{(k - y)^{y + 1}}{(k^4 - y - 1)^y} $, then

\[g'(y) = \frac{(k - y)^{y + 1}}{(k^4 - y - 1)^y} \pth{\frac{k^4(y + 1) - (k + 2)y - 1}{(k - y)(k^4 - y - 1)} + \log\frac{k^4 - y - 1}{k - y}} \gs 0\]

\noindent for $ 2 \ls y \ls k - 1 $. Hence

\[1 - \frac{(k - r)^{r + 1}}{(k^4 - r - 1)^r} \gs 1 - \frac{(k - 2)^{3}}{(k^4 - 3)^2} \gs \frac{999}{1000},\]

\noindent implying the required result holds.
\end{proof}

For any $\mathcal{F}\subseteq{[n]\choose k}$ (not necessary to be $t$-intersecting), a subset $T$ of $[n]$ is called to be a $t$-\emph{cover} of $\mathcal{F}$ if $|T\cap F|\gs t$ for all $F\in\mathcal{F}$, and the $t$-\emph{covering number} $\tau_t(\mathcal{F})$ of $\mathcal{F}\subseteq {[n]\choose k}$ is the minimum size of a $t$-cover of $\mathcal{F}$.
 Note that $ t \ls \tau_t(\CF) \ls k $ and $ \CF $ is trivial if and only if $ \tau_t(\CF) = t $. Denote $ \CC $ to be the family of all $ t $-covers of $ \CF $.

In the rest of this section, we assume $ r = 2 $.

\begin{lem}\label{upperboundlemma3.4}
{\rm (\cite[Lemma 3.4]{CLW2021})} Let $ 1 \ls t \ls k - 2 $, $ n \gs k^4 $ and $ \CF \subseteq \binom{[n]}{k} $ be a maximal $ t $-intersecting family with $ t + 2 \ls \tau_t(\CF) \ls k $, then $ \abs{\CF} \ls k^2\binom{t + 2}{2}\binom{n - t - 2}{k - t - 2} $.
\end{lem}

\noindent{\bf Remark} 
In \cite{CLW2021}, the lower of $n$ and the upper bound of $ \abs{\CF}$ in Lemma \ref{upperboundlemma3.4} are $\binom{t + 2}{2}(k - t + 1)^2 + t $ and $(k - t + 1)^2\binom{t + 2}{2}\binom{n - t - 2}{k - t - 2} $, respectively. Here we relax them for convenience.

\begin{lem}\label{Cao}
{\rm (\cite[Lemmas 2.1, 2.4, 3.3]{CLW2021})} Let $ 1 \ls t \ls k - 2 $, $ n \gs k^4 $ and $ \CF \subseteq \binom{[n]}{k} $ be a maximal $ t $-intersecting family with $ \tau_t(\CF) = t + 1 $, then $ \CC $ is a $ t $-intersecting family and $t\ls \tau_t(\CC)\ls t + 1$. If $ \tau_t(\CC) = t + 1 $, then $ \CC \cong \binom{[t + 2]}{t + 1} $ and $ \CF \cong \CG'_{r,t} $. If $ \tau_t(\CC) = t $, then exactly one of the followings holds.

\begin{enumerate}
\item\label{case1} $ \CC \cong \curlybraces{[t + 1]} $, and

\[\abs{\CF} \ls \binom{n - t - 1}{k - t - 1} + (t + 1)(k - t)(k - t + 1)\binom{n - t - 2}{k - t - 2}.\]
	
\item\label{case2} $ \CC \cong \curlybraces{C \in \binom{[t + 2]}{t + 1}: [t] \subseteq C} $, and

\[\abs{\CF} \ls 2\binom{n - t - 1}{k - t - 1} + (k - 1)(k - t + 1)\binom{n - t - 2}{k - t - 2}.\]

\item\label{case3} $ \CC \cong \curlybraces{C \in \binom{[\ell]}{t + 1}: [t] \subseteq C} $ for some $ t + 3 \ls \ell \ls k + 1 $, and
		
\[\abs{\CF} \ls (\ell - t)\binom{n - t - 1}{k - t - 1} + (k - \ell + 1)(k - t + 1)\binom{n - t - 2}{k - t - 2} + t\binom{n - \ell}{k - \ell + 1} .\]

\end{enumerate}
\noindent Moreover, for cases {\rm (\ref{case1})(\ref{case2})(\ref{case3})} above, for any $ F \in \CF $, if $ [t] \not\subseteq F $, then $ \abs{F \cap [t]} = t - 1 $.
\end{lem}

For convenience, we relax the upper bound of $\abs{\CF}$ in Lemma \ref{Cao} and obtain the following lemma.

\begin{lem}\label{upperboundlemma3.3}
 Let $ 1 \ls t \ls k - 2 $, $ n \gs k^4 $ and $ \CF \subseteq \binom{[n]}{k} $ be a maximal $ t $-intersecting family with $ \tau_t(\CF) = t + 1 $, then $ \CC $ is a $ t $-intersecting family and $t\ls \tau_t(\CC)\ls t + 1$. If $ \tau_t(\CC) = t + 1 $, then $ \CC \cong \binom{[t + 2]}{t + 1} $ and $ \CF \cong \CG'_{r,t} $. If $ \tau_t(\CC) = t $, then exactly one of the followings holds.

\begin{enumerate}
\item\label{case1} $ \CC \cong \curlybraces{[t + 1]} $, and

\[\abs{\CF}  \ls \frac{6}{5}\binom{n - t - 1}{k - t - 1}.\]
	
\item\label{case2} $ \CC \cong \curlybraces{C \in \binom{[t + 2]}{t + 1}: [t] \subseteq C} $, and

\[\abs{\CF} \ls \frac{21}{10}\binom{n - t - 1}{k - t - 1}.\]

\item\label{case3} $ \CC \cong \curlybraces{C \in \binom{[\ell]}{t + 1}: [t] \subseteq C} $ for some $ t + 3 \ls \ell \ls k + 1 $, and
		
\[\abs{\CF}  \ls (k + 1)\binom{n - t - 1}{k - t - 1}.\]

\end{enumerate}
\noindent Moreover, for cases {\rm (\ref{case1})(\ref{case2})(\ref{case3})} above, for any $ F \in \CF $, if $ [t] \not\subseteq F $, then $ \abs{F \cap [t]} = t - 1 $.

\end{lem}

\begin{proof}
 Note that

\[\binom{n - t - 2}{k - t - 2} \times \binom{n - t - 1}{k - t - 1}^{-1} = \frac{k - t - 1}{n - t - 1} \ls \frac{k - 2}{n - 2} \ls \frac{k - 2}{k^4 - 2}.\]

\begin{enumerate}
\item Fix $ k $, we define $ f(x) := (x + 1)(k - x)(k - x + 1) $, then
	
\[\curlybraces{x: f'(x) = 0} = \curlybraces{x_1 = \frac{1}{3}\pth{2k - \sqrt{k^2 + 3k + 3}}, x_2 = \frac{1}{3}\pth{2k + \sqrt{k^2 + 3k + 3}}}.\]
	
\noindent Note that $ x_2 \gs k - 2 $. So

\begin{equation*}
\max\limits_{1 \ls x \ls k - 2}f(x) = \begin{cases}
f(1) & k = 3, 4 \\
f(x_1) & k \gs 5.	
\end{cases}
\end{equation*}

\noindent Thus we have
	
\[\frac{(k - 2)\max\limits_{1 \ls x \ls k - 2}f(x)}{k^4 - 2} \ls \frac{1}{5},\] and we are done by Lemma 2.3 (\ref{case1}).
	
\item Note that
	
\[(k - 1)(k - t + 1)\frac{k-2}{k^4 - 2} \ls \frac{k(k - 1)(k - 2)}{k^4 - 2} \ls \frac{1}{10}.\]
Thus the result holds by Lemma 2.3 (\ref{case2}).
	
\item Note that
	
\[\binom{n - \ell}{k - \ell + 1} \ls \binom{n - \ell + 1}{k - \ell + 1} \ls \binom{n - t - 2}{k - t - 2},\]
	\noindent and
	
\[((k - \ell + 1)(k - t + 1) + t)\frac{k - 2}{k^4 - 2} \ls ((k - 3)k + k - 2)\frac{k - 2}{k^4 - 2} \ls \frac{7}{100}.\]	
Thus the result holds by Lemma 2.3 (\ref{case3}).	
\end{enumerate}\end{proof}

\section{Proof of Theorem \ref{maintheorem}}\label{proof}
Let $ \CF \subseteq \binom{[n]}{k} $ be a maximal $ t $-intersecting family and $ \CC $  be the family of all $ t $-covers of $ \CF $. We will prove Theorem \ref{maintheorem}  according to $\tau_t(\CF)$.

{\bf Case 1.} $ t + 2 \ls \tau_t(\CF) \ls k $. By Lemma \ref{upperboundlemma3.4}, when $ n \gs k^4 $, we have

\[N_{3,t}(\CF) \ls \binom{\abs{\CF}}{3} \ls \frac{1}{6}\abs{\CF}^3 \ls \frac{1}{6}k^6\binom{t + 2}{2}^3\binom{n - t - 2}{k - t - 2}^3.\]

\noindent Since

\[\binom{t + 2}{2}^3 \times \binom{t + 2}{3}^{-2} = \frac{9(t + 1)(t + 2)}{2t^2} \ls 27,\]

\noindent and

\[\binom{n - t - 2}{k - t - 2} \times \binom{n - t - 2}{k - t - 1}^{-1} = \frac{k - t - 1}{n - k},\]

\noindent by Lemma 2.1, we have

\begin{align*}
\frac{N_{3,t}(\CF)}{N_{3,t}(\CG_{2,t})} &\ls \frac{1}{6}k^6 \times 27\binom{t + 2}{3}^2 \pth{\frac{k - t - 1}{n - k}}^3 \binom{n - t - 2}{k - t - 1}^3 N_{3,t}(\CG_{2,t})^{-1} \\
&\ls \frac{500}{111}k^6\binom{t + 2}{3}\left(\frac{k - t - 1}{n - k}\right)^3 \\
&\ls \frac{500}{111}k^6\binom{k}{3}\left(\frac{k - t - 1}{n - k}\right)^3 \\
&\ls \frac{500}{666}k^9\pth{\frac{k - 2}{k^4 - k}}^3= \frac{500}{666}\pth{\frac{k^2(k - 2)}{k^3 - 1}}^3\\
&< 1,
\end{align*} and we are done.

{\bf Case 2.} $ \tau_t(\CF) = t + 1 $.

In this case, we have $ t\ls \tau_t(\CC) \ls  t+1 $ by Lemma \ref{upperboundlemma3.3}. If $ \tau_t(\CC) = t+1 $, then $ \CF \cong \CG'_{r,t} $ by Lemma \ref{upperboundlemma3.3} and we are done. So we just consider
 $ \tau_t(\CC) = t $. From Lemma \ref{upperboundlemma3.3}, we will discuss the following three subcases.

{\bf Case 2.1.} $ \CC \cong \curlybraces{[t + 1]} $. For each $ i \in [t + 1] $, denote $ \CF_i := \curlybraces{F \in \CF: i \notin F} $. Without loss of generality, we may assume $ \CF_1 $ contains the most elements among $ \CF_i $. Fix $ F_2 = ([t + 1] \setminus \{2\}) \cup X \in \CF_2 $, where $ \abs{X} = k - t $ and $ X \cap [t + 1] = \emptyset $. Choose $ F_1 \in \CF_1 $. Because $ [2] \cap F_1 \cap F_2 = \emptyset $ and $ [3, t + 1] \subseteq F_1 \cap F_2 $, due to $ \CF $ is a $ t $-intersecting family, $ F_1 \cap X \ne \emptyset $. Let $ x \in F_1 \cap X $. Because $ S := [2, t + 1] \cup \{x\} $ is not a $t$-cover, there exists some $ F_x \in \CF $ such that $ \abs{F_x \cap S} \ls t - 1 $. We claim that $ \abs{F_x \cap [2, t + 1]} = t - 1 $ and $ x \notin F_x $, here are the reasons.

\begin{itemize}
\item If $ \abs{F_x \cap [2, t + 1]} \ls t - 2 $, then $ \abs{F_x \cap [t + 1]} \ls t - 1 $, a contradiction with $ [t + 1] $ being a $t$-cover.
	
\item If $ x \in F_x $, then by $ \abs{F_x \cap S} \ls t - 1 $, we have $ \abs{F_x \cap [2, t + 1]} \ls t - 2 $, a contradiction.
\end{itemize}

\noindent Let $ F_x = ([t + 1] \setminus \{a\}) \cup Y_x $, where $ a \in [2, t + 1] $, $ \abs{Y_x} = k - t $ and $ Y_x \cap ([t + 1] \cup \{x\}) = \emptyset $. Since $ \CF $ is a $ t $-intersecting family, we have $ \abs{F_1 \cap F_x} \gs t $, implying that $ F_1 \cap Y_x \ne \emptyset $. Thus

\begin{align*}
\abs{\CF_1} &\ls \sum_{x \in X}\sum_{y \in Y_x}\abs{\curlybraces{F \in \CF_1: \{x, y\} \subseteq F}}\\
&\ls \sum_{x \in X}\sum_{y \in Y_x}\binom{n - t - 3}{k - t - 2} \\
&\ls (k - t)^2\binom{n - t - 3}{k - t - 2} \\
&\ls (k - 1)^2\binom{n - t - 3}{k - t - 2}.
\end{align*}

\noindent Note that, for every $(3,t)$-triangle, there exist two distinct indices $ i $ and $ j $ in $ [t + 1] $, such that this $(3,t)$-triangle contains two sets, one from $ \CF_i $ and the other from $ \CF_j $. By Lemma \ref{upperboundlemma3.3}(\ref{case1}), when $ n \gs k^4 $, we have
	
\[N_{3,t}(\CF) \ls \binom{t + 1}{2}\abs{\CF_i}\abs{\CF_j}\abs{\CF} \ls \frac{6}{5}\binom{t + 1}{2} (k - 1)^4\binom{n - t - 3}{k - t - 2}^2 \binom{n - t - 1}{k - t - 1}.\]

\noindent Since

\begin{equation}\label{key1}
\binom{n - t - 3}{k - t - 2} \times \binom{n - t - 2}{k - t - 1}^{-1} = \frac{k - t - 1}{n - t - 2} \ls \frac{k - 2}{n - 3} \ls \frac{k - 2}{k^4 - 3},
\end{equation}
\noindent and
\begin{equation}\label{key2}
\binom{n - t - 1}{k - t - 1} \times \binom{n - t - 2}{k - t - 1}^{-1} = \frac{n - t - 1}{n - k} \ls \frac{n - 2}{n - k} \ls \frac{k^4 - 2}{k^4 - k},
\end{equation}

\noindent and

\[\binom{t + 1}{2} \ls \binom{t + 2}{3},\]

\noindent then by Lemma 2.1,
\begin{align*}
\frac{N_{3,t}(\CF)}{N_{3,t}(\CG_{2,t})} &\ls  \frac{6}{5} (k - 1)^4 \pth{\frac{k - 2}{k^4 - 3}}^2 \pth{\frac{k^4 - 2}{k^4 - k}} \binom{t + 2}{3} \binom{n - t - 2}{k - t - 1}^3 N_{3,t}(\CG_{2,t})^{-1} \\
&\ls\frac{400}{333} (k - 1)^4 \pth{\frac{k - 2}{k^4 - 3}}^2 \pth{\frac{k^4 - 2}{k^4 - k}} \\
&<\frac{k-1}{k}\times\frac{2(k^4-2)(k-1)^2}{(k^4-3)^2}\times\frac{(k-1)(k-2)^2}{k^3-1}\\
&< 1,
\end{align*}
due to $k\gs 3$.

{\bf Case 2.2.} $ \CC \cong \curlybraces{C \in \binom{[t + 2]}{t + 1}: [t] \subseteq C} $. As for any $ a \in [t] $, $ [t + 2] \setminus \{a\} $ is not a $t$-cover, we claim that there exists $ F_0 \in \CF $ with $ F_0 \cap \{t + 1, t + 2\} = \emptyset $. Otherwise, for any $ F' \in \CF $, we have $ \{t + 1, t + 2\} \cap F' \ne \emptyset $. Then it will lead to the following two cases.

\begin{itemize}
\item If $ \{t + 1, t + 2\} \subseteq F' $, then $ \abs{F' \cap [t]} \gs t - 1 $, due to $ [t] \cup \{t + 1\} $ and $ [t] \cup \{t + 2\} $ are $ t $-covers.
	
\item If $ t + 1 \notin F' $, $ t + 2 \in F' $ (resp. $ t + 1 \in F' $, $ t + 2 \notin F' $), then $ [t] \subseteq F' $, due to $ [t] \cup \{t + 1\} $ (resp. $ [t] \cup \{t + 2\} $) is a $ t $-cover.
\end{itemize}

\noindent No matter in which case, for any $ a \in [t] $, we have $ \abs{([t + 2] \setminus \{a\}) \cap F'} \gs t $ for any $ F' \in \CF $, a contradiction with $ [t + 2] \setminus \{a\} $ being not a $t$-cover. Hence our claim holds. Set $ F_0 = [t] \cup X $, where $ \abs{X} = k - t $ and $ [t + 2] \cap X = \emptyset $. Note that every $(3,t)$-triangle must contain one $ F \in \CF $ with $ \abs{F \cap [t]} = t - 1 $. Since $ [t] \cup \{t + 1\} $ and $ [t] \cup \{t + 2\} $ are $t$-covers, we have $ \{t + 1, t + 2\} \subseteq F $. Because $ \CF $ is a $ t $-intersecting family, we have $ \abs{F \cap F_0} \gs t $ which implies $ X \cap F \ne \emptyset $. Hence the number of possible $F$s is at most $ t(k - t)\binom{n - t - 3}{k - t - 2} $. By Lemma \ref{upperboundlemma3.3}(\ref{case2}), when $ n \gs k^4 $, we have
	
\[N_{3,t}(\CF) \ls t(k - t)\binom{n - t - 3}{k - t - 2}\binom{\abs{\CF}}{2} \ls \frac{441}{200} \times \frac{k^2}{4} \binom{n - t - 3}{k - t - 2}\binom{n - t - 1}{k - t - 1}^2.\]

\noindent By (\ref{key1}), (\ref{key2}) and Lemma 2.1, we have
\begin{align*}
\frac{N_{3,t}(\CF)}{N_{3,t}(\CG_{2,t})} &\ls \frac{441}{800} \times k^2 \times \left(\frac{k - 2}{k^4 - 3}\right)\times\pth{\frac{k^4 - 2}{k^4 - k}}^2 \times \binom{n - t - 2}{k - t - 1}^3N_{3,t}(\CG_{2,t})^{-1} \\
&\ls \frac{245}{444}\times \left(\frac{k - 2}{k^4 - 3}\right)\times\pth{\frac{k^4 - 2}{k^3 - 1}}^2\\
&< 1.
\end{align*}	
{\bf Case 2.3.} $ \CC \cong \curlybraces{C \in \binom{[\ell]}{t + 1}: [t] \subseteq C} $ for some $ t + 3 \ls \ell \ls k + 1 $. Note that every $(3,t)$-triangle must contain a set $ F $ satisfying $ \abs{F \cap [t]} = t - 1 $ and $ [t + 1, \ell] \subseteq F $, and the number of possible $F$s is at most $ t\binom{n - \ell}{k - \ell + 1} $. By Lemma \ref{upperboundlemma3.3}(\ref{case3}), when $ n \gs k^4 $, we have
	
\[N_{3,t}(\CF) \ls t\binom{n - \ell}{k - \ell + 1}\binom{\abs{\CF}}{2} \ls \frac{1}{2} t (k + 1)^2\binom{n - t - 3}{k - t - 2}\binom{n - t - 1}{k - t - 1}^2.\]

\noindent Since

\[t\binom{t + 2}{3}^{-1} = \frac{6}{(1 + t)(2 + t)} \ls 1,\]

\noindent By (\ref{key1}), (\ref{key2}) and Lemma 2.1, we have
\begin{align*}
\frac{N_{3,t}(\CF)}{N_{3,t}(\CG_{2,t})} &\ls \frac{1}{2} t (k + 1)^2  \left(\frac{k - 2}{k^4 - 3}\right)\pth{\frac{k^4 - 2}{k^4 - k}}^2 \binom{n - t - 2}{k - t - 1}^3N_{3,t}(\CG_{2,t})^{-1} \\
&\ls \frac{500}{999} \times\frac{(k + 1)^2(k - 2)(k^4 - 2)^2}{(k^4 - 3)(k^4 - k)^2} \\
&< 1.
\end{align*}
This finishes the proof of Theorem \ref{maintheorem}.

\section{Proof of Theorem \ref{maintheorem_gen}}\label{section_gen}

Let $ \CF \subseteq \binom{[n]}{k} $ be an $ r $-wise $ t $-intersecting family with $r\gs 3$ and $1\ls t\ls k-r$. We first have the following lemmas.

\begin{lem}\label{gen_of_prop2.2}
	
If $ \CF \subseteq \binom{[n]}{k} $ is $ r $-wise $ t $-intersecting with $ \tau_t(\CF) = s $, then $ \CF $ is $ \sqb{(r - 2)(s - t) + t} $-intersecting.
	
\end{lem}

\begin{proof}
	
Recall $ r \gs 3 $. By way of contradiction, suppose there exist $ F', F'' \in \CF $ with $ \abs{F' \cap F''} \ls (r - 2)(s - t) + t - 1 $. Denote $ X := F' \cap F'' $, then $ X $ admits an arbitrary partition: $ X = Y_0 \cup X_1 \cup \cdots \cup X_{r - 3} $, where $ \abs{Y_0} \ls s - 1 $ and for $ i \in [r - 3] $, $ \abs{X_i} \ls s - t $. Since $ \abs{Y_0} < \tau_t(\CF) $, $ Y_0 $ is not a $t$-cover of $\CF$. Then there exists $ F_1 \in \CF $ such that $ \abs{Y_0 \cap F_1} \ls t - 1 $. Similarly, for $ i \in [r - 3] $, we can inductively obtain the pairs $ (Y_i, F_{i + 1}) $ such that $ Y_i = (Y_{i - 1} \cap F_i) \cup X_i $ and $ F_{i + 1} $ is an element in $ \CF $ with $ \abs{Y_i \cap F_{i + 1}} \ls t - 1 $. Then
\begin{align*}
&\abs{F' \cap F'' \cap F_1 \cap F_2 \cap \cdots \cap F_{r - 2}} \\
=&\abs{(Y_0 \cup X_1 \cup X_2 \cup \cdots \cup X_{r - 3}) \cap F_1 \cap F_2 \cap \cdots F_{r - 2}} \\
\ls&\abs{((Y_0 \cap F_1) \cup X_1 \cup X_2 \cup \cdots \cup X_{r - 3}) \cap F_2 \cap \cdots \cap F_{r - 2}} \\
=&\abs{(Y_1 \cup X_2 \cup \cdots \cup X_{r - 3}) \cap F_2 \cap \cdots \cap F_{r - 2}} \\
\ls&\abs{Y_{r - 3} \cap F_{r - 2}} \\
\ls&t - 1.
\end{align*}
\noindent This is impossible since $ \CF $ is $ r $-wise $ t $-intersecting.
\end{proof}

\begin{lem}\label{the_bound}
	
For $ \ell \gs t $, if $ \CF \subseteq \binom{[n]}{k} $ is $ \ell $-intersecting with $ \tau_t(\CF) \gs s $, then

\[\abs{\CF} \ls \pth{\frac{k - s + 2}{\ell - s + 2}}^{s - 1}\binom{k}{\ell}\binom{n - \ell - s + t}{k - \ell - s + t}.\]

\noindent In particular, if $ \CF $ is $ r $-wise $ t $-intersecting with $ \tau_t(\CF) = t + 1 $, then

\[\abs{\CF} \ls \pth{\frac{k - t + 1}{r - 1}}^t \binom{k}{r + t - 2} \binom{n - r - t + 1}{k - r - t + 1} \ls k^{r + 2t - 2}\binom{n - r - t + 1}{k - r - t + 1}.\]

\end{lem}

\begin{proof}
	
Fix some $ G \in \CF $. By counting the elements in $ \{(x, F) \in \binom{G}{s - 1} \times \CF: x \subseteq F\} $ in two ways, we have
$$\abs{\CF} \ls \frac{1}{\binom{\ell}{s - 1}}\sum_{x \subseteq G, \abs{x} = s - 1} \abs{\curlybraces{F \in \CF: x \subseteq F}}.$$
\noindent Since $ \abs{x} < \tau_t(\CF) $, there exists $ G_x \in \CF $ such that $ \abs{G_x \cap x} \ls t - 1 $. For any $ F \supseteq x $, because $ \CF $ is $ \ell $-intersecting, we have $ \abs{F \cap G_x} \gs \ell $, which implies $ F $ contains at least $ (\ell - t + 1) $ elements of $ G_x $ outside $ x $. Hence

\[\abs{\curlybraces{F \in \CF: x \subseteq F}} \ls \abs{\bigcup_{y \subseteq G_x, \abs{y} = \ell}\curlybraces{F \in \CF: x \cup y \subseteq F}} \ls  \binom{k}{\ell}\binom{n - \ell - s + t}{k - \ell - s + t}.\]
Thus we have
\begin{align*}
\abs{\CF} &\ls \frac{1}{\binom{\ell}{s - 1}}\sum_{x \subseteq G, \abs{x} = s - 1} \abs{\curlybraces{F \in \CF: x \subseteq F}} \\
&\ls \frac{\binom{k}{s - 1}}{\binom{\ell}{s - 1}}\binom{k}{\ell}\binom{n - \ell - s + t}{k - \ell - s + t}\\
&\ls\pth{\frac{k - s + 2}{\ell - s + 2}}^{s - 1}\binom{k}{\ell}\binom{n - \ell - s + t}{k - \ell - s + t}.
\end{align*}

\noindent Applying Lemma \ref{gen_of_prop2.2}, we obtain the second claim.
\end{proof}

\begin{lem}\label{tau_bigger_than_t+2}
	
If $ \CF \subseteq \binom{[n]}{k} $ is $ r $-wise $ t $-intersecting with $ \tau_t(\CF) \gs t + 2 $, then there exists two positive constants $ c = c(r, t) $ and $ d = d(r, t) $, such that $ N_{r+1,t}(\CF) < N_{r+1,t}(\CG_{r,t}) $ for $ n \gs c k^d $.

\end{lem}

\begin{proof}
	
By Lemma \ref{gen_of_prop2.2}, $ \CF $ is $ (2r + t - 4) $-intersecting. By Lemma \ref{the_bound},

\[ \abs{\CF} \ls \pth{\frac{k - t}{2r - 4}}^{t + 1}\binom{k}{2r + t - 4}\binom{n - 2r - t + 2}{k - 2r - t + 2} \ls k^{2r + 2t - 3}\binom{n - 2r - t + 2}{k - 2r - t + 2}.\]

\noindent Hence,

\[N_{r+1,t}(\CF) \ls \binom{\abs{\CF}}{r + 1} \ls \frac{1}{(r + 1)!}\abs{\CF}^{r + 1} \ls \frac{1}{(r + 1)!}\pth{k^{2r + 2t - 3}\binom{n - 2r - t + 2}{k - 2r - t + 2}}^{r + 1}.\]

\noindent Assume $ n \gs 2k $. Note that

\begin{align*}
\binom{n - r - t}{k - r - t + 1} \times \binom{n - 2r - t + 2}{k - 2r - t + 2}^{-1}
&= \frac{n-k}{k - r - t + 1} \times \frac{(n - r - t)\cdots(n - 2r - t + 3)}{(k - r - t)\cdots(k - 2r - t + 3)} \\
&\gs \pth{\frac{n}{k}}^{r - 2}.
\end{align*}

\noindent The result holds by taking $ c = \max\curlybraces{\pth{\frac{1000(t - 1)!}{999(r + t)!}}^{\frac{1}{(r + 1)(r - 2)}}, 2} $ and $ d = \max\curlybraces{\frac{3r + 2t - 5}{r - 2}, 1} $.
\end{proof}

If $ \tau_t(\CF) = t  $, that is $ \CF $ is a trivial $ r $-wise $ t $-intersecting family, then $ N_{r+1,t}(\CF) = 0 $. If $ \tau_t(\CF) \gs t + 2 $, then $ N_{r+1,t}(\CF) < N_{r+1,t}(\CG_{r,t}) $ by Lemma \ref{tau_bigger_than_t+2}. Thus, in order to complete the proof, we may assume $ \tau_t(\CF) = t + 1 $ and just consider $t\gs 2$ by Theorem \ref{npthm2} in the following. Recall $ \CC $ is the family of all $ t $-covers of $ \CF $. Now we construct an auxiliary $ (t + 1) $-uniform hypergraph $ G = (V, E) $ as following: $ E(G) =\{C~|~C\in \CC,~|C|=t+1\} $ and $ V(G) = \bigcup_{C \in E} C $. By Lemma \ref{upperboundlemma3.3},
 $ \CC $ is an $ t $-intersecting family which implies for any $C_1,C_2\in E(G)$, $|C_1\cap C_2|=t$. Assume $H_1,\ldots,H_\ell$ are the components of $G$, where $\ell\gs 1$. Then $|V(H_i)|\gs t+1$ for all $1\ls i\ls \ell$. A component $H_i$ is called a {\em clique} if $E(H_i)=\binom{V(H_i)}{t+1}$, that is all $(t+1)$-subsets of $V(H_i)$ are the hyperedges of $H_i$.

\begin{lem}\label{clique_lemma}
	
If $ G $ contains a component which is not a clique, then there exists two positive constants $ c = c(r, t) $ and $ d = d(r, t) $, such that $ N_{r+1,t}(\CF) < N_{r+1,t}(\CG_{r,t}) $ for $ n \gs c k^d $.	

\end{lem}

\begin{proof} Assume that $H_1$ is not a clique. Then $|V(H_1)|\gs t+2$. Without loss of generality, we may assume that $[t+2]\subseteq V(H_1)$. 	
Denote $ S_i := [t + 2] \setminus \{i\} $.  Without loss of generality, we may assume that $ S_1 ,S_2 \in E(H_1)$ but $ S_3\notin E(H_1) $. Since $S_1$ is a $t$-cover of  $ \CF $, we have $ \abs{F \cap [3, t + 2]} \gs t - 1 $ for all $F\in \CF $.
Thus
an $ (r + 1,t) $-triangle must contain a set $ F' $ such that $ \abs{F' \cap [3, t + 2]} = t - 1 $. Then $ [2] \subseteq F' $ as $ S_1 $ and $ S_2 $ are $t$-covers. Now we claim that there exists a set $ F_0 \in \CF $ such that $ F_0 \cap [2] = \emptyset $. Suppose that $ F \cap [2] \ne \emptyset $ for each $ F\in  \CF$.

\begin{itemize}
\item There is  $ F\in  \CF$ such that $ [2] \subseteq F $. Since $ S_1 $ is a $t$-cover, $ \abs{F \cap S_1} \gs t $ which implies $ \abs{F \cap [4, t + 2]} \gs t - 2 $. Then $ \abs{F \cap S_3} \gs t $.
	
\item Suppose $ F \cap [2] = \{1\} $ or $ F \cap [2] = \{2\} $, say $ F \cap [2] = \{1\} $. Since $ S_1 $ is a $t$-cover, $ [3, t + 2] \subseteq F $, which implie $ S_2 \subseteq F $. Then $ \abs{F \cap S_3} = t $.
\end{itemize}

\noindent No matter in which case, $ \abs{F \cap S_3} \gs t $ for  each $ F\in  \CF$, a contradiction with $ S_3 $ being not a $t$-cover. Hence our claim holds. Set $ F_0 = [3, t + 2] \cup X $ with $ \abs{X} = k - t $ and $ X \cap [t + 2] = \emptyset $. By Lemma \ref{gen_of_prop2.2}, $ \CF $ is $ (r + t - 2) $-intersecting, then $ \abs{F' \cap X} \gs r - 1 $. Hence

\[\abs{\curlybraces{F' \in \CF: \abs{F' \cap [3, t + 2]} = t - 1}} \ls \binom{k - t}{r - 1}\binom{n - r - t - 1}{k - r - t} \ls k^{r - 1}\binom{n - r - t - 1}{k - r - t}.\]

\noindent By Lemma \ref{the_bound}, we have
\[N_{r+1,t}(\CF) \ls k^{r - 1}\binom{n - r - t - 1}{k - r - t}\abs{\CF}^r \ls k^{r - 1}\binom{n - r - t - 1}{k - r - t}\pth{k^{r + 2t - 2}\binom{n - r - t + 1}{k - r - t + 1}}^r.\]
By Lemma 2.1,
$$N_{r+1,t}(\CG_{r,t}) \gs \frac{999}{1000}\binom{r + t}{r + 1}\binom{n - r - t}{k - r - t + 1}^{r + 1}.$$

\noindent Assume $ n \gs 2k $. Note that
 we have
\[\binom{n - r - t}{k - r - t + 1} \times \binom{n - r - t + 1}{k - r - t + 1}^{-1} = \frac{n - k}{n - r - t + 1} \gs \frac{1}{2},\]
\noindent and
\[\binom{n - r - t}{k - r - t + 1} \times \binom{n - r - t - 1}{k - r - t}^{-1} = \frac{n - r - t}{k - r - t + 1} \gs \frac{1}{2}\frac{n}{k}.\]\noindent When $ c = \max\curlybraces{\frac{1000 \times 2^{r + 1}}{999} \binom{r + t}{r + 1}^{-1}, 2} $, $ d = r(r + 2t - 1) $ and $ n \gs c k^d $, we have $ N_{r+1,t}(\CF) < N_{r+1,t}(\CG_{r,t}) $.
\end{proof}

By Lemma \ref{clique_lemma}, all components of $G$ are cliques.

\begin{lem}\label{bigger_lemma}
	
If $ G $ contains a  component  of order larger than $ r + t $, then there exists two positive constants $ c = c(r, t) $ and $ d = d(r, t) $, such that $ N_{r+1,t}(\CF) < N_{r+1,t}(\CG_{r,t}) $ for $ n \gs c k^d $.	

\end{lem}

\begin{proof}
	
Without loss of generality, we may assume $ H_1 $ contains a subclique of size $ r + t + 1 $, say $ ([r + t + 1], \binom{[r + t + 1]}{t + 1}) $. For each $ F \in \CF $ with $ \abs{F \cap [t]} = t - 1 $, since $ [t] \cup \{i\} $ is a $ t $-cover of $ \CF $ for every $ i \in [t + 1, r + t + 1] $, we have $ [t + 1, r + t + 1] \subseteq F $. Note that $ \abs{\curlybraces{F \in \CF: \abs{F \cap [t]} = t - 1}} \ls t \binom{n - r - t - 1}{k - r - t} $. Since any $ (r + 1,t) $-triangle must contain a set $ F \in \CF $ with $ \abs{F \cap [t]} = t - 1 $, by Lemma \ref{the_bound}, we have

\[N_{r+1,t}(\CF) \ls t \binom{n - r - t - 1}{k - r - t}\abs{\CF}^r \ls t \binom{n - r - t - 1}{k - r - t}\pth{k^{r + 2t - 2}\binom{n - r - t + 1}{k - r - t + 1}}^r.\]

\noindent By taking $ c = \max\curlybraces{\frac{1000\times 2^{r + 1} t}{999} \binom{r + t}{r + 1}^{-1}, 2} $ and $ d = r(r + 2t - 2) + 1 $, the required result holds.
\end{proof}

\begin{lem}\label{smaller_lemma}
	
If $ G $ contains a  component  of order smaller than $ r + t $, then there exists two positive constants $ c = c(r, t) $ and $ d = d(r, t) $, such that $ N_{r+1,t}(\CF) < N_{r+1,t}(\CG_{r,t}) $ for $ n \gs c k^d $.	

\end{lem}

\begin{proof}
	
Suppose this clique has vertex set $ [m] $, where $ t + 1 \ls m \ls r + t - 1 $. For each $ i \in [m] $, denote $ \CF_i := \{F \in \CF: i \notin F\} $. Without loss of generality, we may assume $ \CF_1 $ contains the most elements among $ \CF_i $. Fix some $ F_2 = ([m] \setminus \{2\}) \cup X \in \CF_2 $, where $ \abs{X} = k - m + 1 $ and $ X \cap [m] = \emptyset $. Choose $ F_1 \in \CF_1 $. Since $ [2] \cap F_1 \cap F_2 = \emptyset $ and $ [3, m] \subseteq F_1 \cap F_2 $, by $ \CF $ being $ (r + t - 2) $-intersecting, $ F_1 \cap X \ne \emptyset $. Fix $ x \in F_1 \cap X $. Because $ S := [2, t + 1] \cup \{x\} $ is not a $t$-cover, there exists some $ F_x \in \CF $ with $ \abs{F_x \cap S} \ls t - 1 $. We claim that $ \abs{F_x \cap [2, t + 1]} = t - 1 $ and $ x \notin F_x $. Indeed,

\begin{itemize}
\item if $ \abs{F_x \cap [2, t + 1]} \ls t - 2 $, then $ \abs{F_x \cap [t + 1]} \ls t - 1 $, a contradiction with $ [t + 1] $ being a $t$-cover;
	
\item if $ x \in F_x $, then $ \abs{F_x \cap [2, t + 1]} \ls t - 2 $, a contradiction.
\end{itemize}

\noindent Let $ F_x = ([m] \setminus \{a\}) \cup Y_x $, where $ a \in [2, t + 1] $, $ \abs{Y_x} = k - m + 1 $ and $ Y_x \cap ([m] \cup \{x\}) = \emptyset $. By $ \CF $ being $ (r + t - 2) $-intersecting, $ \abs{F_1 \cap F_x} \gs r + t - 2 $, then $ \abs{F_1 \cap Y_x} \gs r + t - m $. Thus

\begin{align*}
\abs{\CF_1} &\ls \sum_{x \in X} \sum_{y \subseteq Y_x, \abs{y} = r + t - m} \abs{\curlybraces{F \in \CF_1: \{x\} \cup y \subseteq F}}\\
&\ls \sum_{x \in X} \sum_{y \subseteq Y_x, \abs{y} = r + t - m}\binom{n - r - t - 1}{k - r - t} \\
&\ls (k - m + 1)\binom{k - m + 1}{r + t - m}\binom{n - r - t - 1}{k - r - t} \\
&\ls k^{r}\binom{n - r - t - 1}{k - r - t}.
\end{align*}

\noindent Since any $ (r + 1,t) $-triangle must contain at least one $ F_i \in \CF_i $ for some $ i \in [m] $, by Lemma \ref{the_bound}, we have

\[N_{r+1,t}(\CF) \ls m\abs{\CF_i}\abs{\CF}^r \ls k^{r + 1}\binom{n - r - t - 1}{k - r - t}\pth{k^{r + 2t - 2}\binom{n - r - t + 1}{k - r - t + 1}}^r.\]

\noindent By taking $ c = \max\curlybraces{\frac{1000 \times 2^{r + 1}}{999}\binom{r + t}{r + 1}^{-1}, 2} $, $ d = r(r + 2t - 1) + 2 $, the required result holds.
\end{proof}

\begin{lem}\label{final_lemma}
	
Suppose $H_1,\ldots, H_\ell$ are cliques of order $ r + t $. If $ \ell\gs 2 $, then $ N_{r+1,t}(\CF) = 0 $.	

\end{lem}

\begin{proof}
	
By hypothesis, we may assume
\[H_1=\pth{[r + t], \binom{[r + t]}{t + 1}}~\mbox{and}~  H_2=\pth{[r + t + 1, 2r + 2t], \binom{[r + t + 1, 2r + 2t]}{t + 1}} .\]
\noindent Then
\[\CF \subseteq \curlybraces{F \in \binom{[n]}{k}: \abs{F \cap [r + t]} \gs r + t - 1, \abs{F \cap [r + t + 1, 2r + 2t]} \gs r + t - 1}.\]
\noindent Hence, for any $ F_1, \cdots, F_{r + 1} \in \CF $, we have
\[\abs{F_1 \cap \cdots \cap F_{r + 1}} \gs 2r + 2t - (2r + 2) = 2t - 2 \gs t.\]	
Thus $ N_{r+1,t}(\CF) = 0 $.
\end{proof}

By Lemmas \ref{clique_lemma} to \ref{final_lemma}, we obtained that unless $G$ consists only one clique of order $ r + t $, the number of $(r + 1,t)$-triangles in $\CF$ is strictly smaller than $N_{r+1,t}(\CG_{r,t}) $. On the other hand, if
$ G $ is a clique of order $ r + t $, then $ \CG_{r,t} \subseteq \CF \subseteq \CG'_{r,t} $. This finishes the proof of Theorem \ref{maintheorem_gen}.

\end{document}